\documentclass[10pt]{article}
\font\smallit=cmti10
\font\smalltt=cmtt10

\usepackage{lipsum}

\newcommand\blfootnote[1]{%
  \begingroup
  \renewcommand\thefootnote{}\footnote{#1}%
  \addtocounter{footnote}{-1}%
  \endgroup
}
\usepackage{color}
\usepackage{hyperref}
\usepackage{enumerate}
\usepackage{color}
\usepackage{breakurl}
\usepackage{comment}
\newcommand{\bburl}[1]{\textcolor{blue}{\url{#1}}}

\usepackage{amssymb,latexsym,amsmath,epsfig,amsthm} 

\makeatletter

\renewcommand\section{\@startsection {section}{1}{\z@}
{-30pt \@plus -1ex \@minus -.2ex}
{2.3ex \@plus.2ex}
{\normalfont\normalsize\bfseries\boldmath}}

\renewcommand\subsection{\@startsection{subsection}{2}{\z@}
{-3.25ex\@plus -1ex \@minus -.2ex}
{1.5ex \@plus .2ex}
{\normalfont\normalsize\bfseries\boldmath}}

\renewcommand{\@seccntformat}[1]{\csname the#1\endcsname. }

\makeatletter
\newcommand{\myitem}[1]{%
\item[#1]\protected@edef\@currentlabel{#1}%
}
\makeatother

\makeatother

\newtheorem{thm}{Theorem}[section]
\newtheorem{conj}[thm]{Conjecture}
\newtheorem{cor}[thm]{Corollary}
\newtheorem{lem}[thm]{Lemma}

\newtheorem{rek}[thm]{Remark}

\newcommand{\Mod}[1]{\ \mathrm{mod}\ #1}

\begin{document}

\begin{center}
\uppercase{\bf Infinite Families of Partitions into MSTD Subsets}
\vskip 20pt
{\bf H\`ung Vi\d{\^e}t Chu}\\
{\smallit Department of Mathematics, Washington and Lee University, Lexington, VA 24450}\\
{\tt chuh19@mail.wlu.edu}\\
\vskip 10pt
{\bf Noah Luntzlara}\\
{\smallit Department of Mathematics, University of Michigan, Ann Arbor, MI 48109}\\
{\tt nluntzla@umich.edu}\\
\vskip 10pt
{\bf Steven J. Miller}\\
{\smallit Department of Mathematics and Statistics, Williams College, Williamstown, MA 01267}\\
{\tt sjm1@williams.edu}\\
\vskip 10pt
{\bf Lily Shao}\\
{\smallit Department of Mathematics and Statistics, Williams College, Williamstown, MA 01267}\\
{\tt ls12@williams.edu}\\
\end{center}
\vskip 20pt
\centerline{\smallit Received: , Revised: , Accepted: , Published: } 
\vskip 30pt

\centerline{\bf Abstract}
\blfootnote{The authors were supported by NSF grants DMS1659037 and DMS1561945, the Finnerty Fund, University of Michigan, Washington and Lee University and Williams College.}
\noindent

A set $A$ is MSTD (more-sum-than-difference) if $|A+A|>|A-A|$. Though MSTD sets are rare, Martin and O'Bryant proved that there exists a positive constant lower bound for the proportion of MSTD subsets of $\{1,2,\ldots ,r\}$ as $r\rightarrow\infty$. Later, Asada et al.\ showed that there exists a positive constant lower bound for the proportion of decompositions of $\{1,2,\ldots,r\}$ into two MSTD subsets as $r\rightarrow\infty$. However, the method is probabilistic and does not give explicit decompositions.

Continuing this work, we provide an efficient method to partition $\{1,2,\ldots,r\}$ (for $r$ sufficiently large) into $k \ge 2$ MSTD subsets, positively answering a question raised by Asada et al. as to whether this is possible for all such $k$. Next, let $R(k)$ be the smallest integer such that for all $r\ge R(k)$, $\{1,2,\ldots,r\}$ can be $k$-decomposed into MSTD subsets. We establish rough lower and upper bounds for $R(k)$. Lastly, we provide a sufficient condition on when there exists a positive constant lower bound for the proportion of decompositions of $\{1,2,\ldots,r\}$ into $k$ MSTD subsets as $r\rightarrow \infty$.

\pagestyle{myheadings}
\markright{\smalltt INTEGERS: 19 (2019)\hfill}
\thispagestyle{empty}
\baselineskip=12.875pt
\vskip 30pt

\section{Introduction}

\subsection{Background}
Given a set $A$ of integers, define $A+A=\{a_i+a_j|a_i,a_j\in A\}$ and $A-A=\{a_i-a_j|a_i,a_j\in A\}$. Then $A$ is said to be sum-dominant or MSTD (more-sums-than-differences) if $|A+A|>|A-A|$, balanced if $|A+A|=|A-A|$ and difference-dominated if $|A+A|<|A-A|$; see \cite{He,Ma,Na1,Na2,Ru1,Ru2,Ru3} for some history and early results in the subject. Research on MSTD sets has made great progress in the last twenty years. In particular, Martin and O'Bryant \cite{MO} showed that with the uniform model, where each element is chosen with probability $1/2$, the proportion of MSTD subsets of $\{1,2,\ldots,r\}$ is uniformly bounded below by a positive constant for large enough $r$. Zhao \cite{Zh2} showed that the proportion converges as $r\to \infty$ and improved the lower bound to $4.28\cdot 10^{-4}$. On the other hand, Hegarty and Miller \cite{HM} proved that with a sparse model, where each element is chosen with probability $p(r)$ such that  $r^{-1} = o(p(r))$ and $p(r)\to 0$ as $r\to\infty$, almost all sets are difference-dominated. These two results do not contradict each other since the probability of being MSTD subsets depends on which model we are using. In proving a lower bound for the proportion of MSTD subsets, Martin and O'Bryant used the probabilistic method and did not give explicit constructions of MSTD sets. Later works gave explicit construction of large families of MSTD sets: Miller et al. \cite{MOS} gave a family of MSTD subsets of $\{1,2,\ldots,r\}$ with density $\Theta(1/r^4)$\footnote{\cite{ILMZ2} showed that with slightly more work, the density is improved to $\Theta(1/r^2)$.}, while Zhao \cite{Zh1} gave a denser family with density $\Theta(1/r)$, the current record.

In \cite{AMMS}, the authors used a technique introduced by Zhao \cite{Zh2} to show that the proportion of 2-decompositions (i.e., partitions into two sets) of $\{1,2,\ldots,r\}$ that gives two MSTD subsets is bounded below by a positive constant. This result is surprising in view of the conventional method of constructing MSTD sets, which is to fix a fringe pair $(L,R)$ of two sets containing elements to be used in the fringe of the interval and argue that all the middle elements appear with some positive probability. (The fringe pair ensures that some of the largest and smallest differences are missed and that our set is MSTD.) However, the result in \cite{AMMS} seems to suggest that we can find two (or more) disjoint fringe pairs $(L_1, R_1)$ and $(L_2,R_2)$ such that $L_1\cup L_2$ and $R_1\cup R_2$ cover a full set of left and right elements of $\{1,2,\ldots,r\}$ and $(L_1,R_1),(L_2,R_2)$ are two fringe pairs for two disjoint MSTD sets. Previous research has focused on each fringe pair independently, so it is interesting to see that two (or more) fringe pairs can complement each other nicely on both sides of $\{1,2,\ldots,r\}$. Motivated by that, we provide a method to construct these fringe pairs and study partitions of $\{1,2,\ldots,r\}$ into MSTD subsets more thoroughly.


\subsection{Notation and Main Results}
Let $[a,b]$ denote $\{\ell\in\mathbb{Z}|a\le \ell\le b\}$ and $I_r$ denote $[1,r]$. We use the idea of $P_n$ sets described in \cite{MOS}. A set $A$ is said to be $P_n$ if the following conditions are met. Let $a=\min A$ and $b=\max A$. Then
\begin{align}
A+A&\ \supseteq \ [2a+n,2b-n] \label{Pnsums}\\
A-A&\ \supseteq \ [(a-b)+n,(b-a)-n].\label{Pndiffs}
\end{align}
A set $A$ is $P_n$ with respect to sums ($SP_n$) if condition~(\ref{Pnsums}) is satisfied, and $P_n$ with respect to differences ($DP_n$) if condition~(\ref{Pndiffs}) is satisfied. Next, let $[a,b]_2$ denote $\{\ell\in\mathbb{Z}|a\le \ell\le b \mbox{ and }\ell-a\mbox{ is even}\}$. Finally, a $2$-decomposition of a set $S$ is $A_1\cup A_2=S$, where $A_1$, $A_2$ are nonempty and $A_1\cap A_2=\emptyset$. We use the words \textit{decomposition} and \textit{partition} interchangeably.
Our main result is the following.

\begin{thm}\label{2decomp}
Let $A_1$ and $A_2$ be chosen such that both are MSTD, $A_1$ is $P_n$ and $A_2$ is $P_{n-4}$. Also,
\begin{enumerate}
\item $(A_1, A_2)$ partition $[1,2n]$,\\
\item $A_i=L_i\cup R_i$ with $L_i\subseteq [1,n]$ and $R_i\subseteq [n+1,2n]$ for $i=1,2$,\\
\item $[1,4]\cup \{n\}\subseteq L_1$ and $\{n+1\}\cup [2n-3,2n]\subseteq R_1$, and\\
\item $[5,7]\subseteq L_2$ and $[2n-6,2n-4]\subseteq R_2$.
\end{enumerate}
(See Remark \ref{Exp1} for an example of such sets $A_1$ and $A_2$.)

Pick $k\ge n/2+2$ and $m\in\mathbb{N}_0$. Set
\begin{align*}
R_1'\ &=\ R_1+m+4k+4,\\
R_2'\ &=\ R_2+m+4k+4,\\
O_{11}\ &=\ \{n+4\}\cup[n+5,n+2k+1]_2\cup\{n+2k+2\},\\
O_{12}\ &=\ \{n+m+2k+3\}\cup[n+m+2k+4,n+m+4k]_2\cup\{n+m+4k+1\},\\
O_{21}\ &=\ [n+1,n+3]\cup[n+6,n+2k]_2\cup [n+2k+3,n+2k+5],\\
O_{22}\ &=\ [n+m+2k,n+m+2k+2]\cup[n+m+2k+5,n+m+4k-1]_2\\
&\qquad\qquad\cup[n+m+4k+2,n+m+4k+4].
\end{align*}
Let $M_1\subseteq [n+2k+6,n+m+2k-1]$ such that within $M_1$, there exists a sequence of pairs of consecutive elements, where consecutive pairs in the sequence are not more than $2k-1$ apart and the sequence starts with a pair in $[n+2k+6,n+4k+1]$ and ends with a pair in $[n+m+4,n+m+2k-1]$. Let $M_2\subseteq [n+2k+6,n+m+2k-1]$ such that within $M_2$, there exists a sequence of triplets of consecutive elements, where consecutive triplets in the sequence are not more than $2k+5$ apart and the sequence starts with a triplet in $[n+2k+6,n+4k+5]$ and ends with a triplet in $[n+m,n+m+2k-1]$. Also, $M_1\cap M_2=\emptyset$ and $M_1\cup M_2=[n+2k+6,n+m+2k-1]$. Then
\begin{align*}
A'_1\ &=\ L_1\cup O_{11}\cup M_1\cup O_{12}\cup R'_1\\
A'_2\ &=\ L_2\cup O_{21}\cup M_2\cup O_{22}\cup R'_2
\end{align*} are both MSTD and partition $[1,2n+m+4k+4]$.
\end{thm}

\begin{rek}\label{Exp1}\normalfont
To show that our family is not empty, we need to show the existence of at least one pair of $A_1$ and $A_2$. Note that our technique is similar to many other papers \cite{He,MO,MOS,MPR,PW} in the sense that we need a good fringe to start with. A random search yielded
\begin{align*}
A_1\ &=\ \{1, 2, 3, 4, 8, 9, 11, 13, 14, 15, 20, 21, 26, 27, 28, 31, 33, 37, 38, 39, 40\},\\A_2\ &=\ \{5,6,7,10,12,16,17,18,19,22,23,24,25,29,30,32,34,35,36\}.\end{align*}
We have
\begin{equation*}
\begin{aligned}
A_1+A_1\ &=\ [2,80]\\
A_1-A_1\ &=\ [-39,39]\backslash\{\pm 21\}
\end{aligned}
\quad \text{and} \quad
\begin{aligned}
A_2+A_2\ &=\ [10,72]\\
A_2-A_2\ &=\ [-31,31]\backslash\{\pm 21\}.
\end{aligned}
\end{equation*} Clearly, $A_1$ is $P_{20}$ and $A_2$ is $P_{16}$. It can be easily checked that all conditions mentioned in Theorem \ref{2decomp} are satisfied. These pairs of sets $A_1$ and $A_2$ are not hard for computers to find: for $n=20$, computer search shows that there are 48 such pairs.
\end{rek}

\begin{rek}\normalfont 
Our method of decomposing $I_r$ into two MSTD sets allows a lot of freedom in choosing the middle elements. This is because once the fringe elements are chosen, the conditions placed on $M_1$ and $M_2$ are relatively weak.
\end{rek}

Next, we answer positively question $(3)$ in \cite{AMMS}, where the authors ask: Can we decompose $\{1,2,\ldots,r\}$ into three sets which are MSTD? For any finite number $k$, is there a sufficiently large $r$ for which there is a $k$-decomposition into MSTD sets?

\begin{thm}\label{kdecomp}
Let $k\in\mathbb{N}_{\ge 2}$ be chosen.
\begin{enumerate}
\item There exists the smallest $R(k)\in\mathbb{N}$ such that for all $r\ge R(k)$, $I_r$ can be $k$-decomposed into MSTD subsets, while $I_{R(k)-1}$ cannot be $k$-decomposed into MSTD subsets.

\item In particular, we find some rough bounds\footnote{We make no attempt to optimize these bounds. Finer analysis may give us better bounds.}:
\begin{enumerate}
\item when $k$ is even, $8k\le R(k)\le 10k$,
\item when $k\ge 5$ odd, $8k\le R(k)\le 20k-14$, and
\item when $k=3$, $24\le R(k)\le 4T+24$,
\end{enumerate}
where $T=\min\{\max A: |A+A|-|A-A|\ge 10|A|\}$.
\end{enumerate}
\end{thm}

We prove Theorem \ref{kdecomp} using sets constructed by the base expansion method\footnote{We can generate an infinite family of MSTD sets from a given MSTD set through the base expansion method. Let $A$ be an MSTD set, and let $A_{k,m}=\{\sum_{i=1}^{k}a_im^{i-1}:a_i\in A\}$. If $m$ is sufficiently large, then $|A_{k,m}\pm A_{k,m}| = |A\pm
A|^k$ and $|A_{k,m}|=|A|^k$.} that helps generate an infinite family of MSTD sets from a given MSTD set. The method is a very powerful tool and has been used extensively in the literature including \cite{He},\cite{ILMZ1} and \cite{ILMZ2}. However, the base expansion method turns out to be inefficient in terms of our MSTD sets' cardinality. Hence, we present a second, more efficient approach by using a particular family of MSTD sets. We present both proofs since they are of independent interest: the first proof is less technical but less efficient. Also, the second proof cannot resolve the case $k=3$ while the first can.

Lastly, we give a sufficient condition on when there exists a positive constant lower bound for the proportion of $k$-decompositions of $I_r$ into MSTD subsets.
The condition offers an alternative proof of Theorem 1.4 in \cite{AMMS} ($k=2$). Due to the condition, we make the following conjecture.

\begin{conj}\label{conjallk}
For any finite $k\ge 2$, the proportion of $k$-decompositions into MSTD subsets is bounded below by a positive constant.
\end{conj}

The outline of the paper is as follows. In Section \ref{explicit2}, we provide an efficient method to decompose $I_r$ into two MSTD subsets; Section \ref{explicitk} presents two methods to decompose $I_r$ into $k\ge 4$ MSTD subsets; Appendix \ref{ApenR} is devoted to establishing the bounds mentioned in Theorem \ref{kdecomp} and the sufficient condition for a positive constant lower bound of the proportion of $k$-decompositions into MSTD subsets in Appendix \ref{Apensuf}. Appendix \ref{Apen34} gives a proof of Theorem \ref{threesetsforpartition}. Appendix \ref{Apenex} contains many examples illustrating our lemmas and theorems. 


\section{Explicit 2-decomposition into MSTD Subsets}\label{explicit2}
In this section, we show how we can decompose $I_r$ into two MSTD subsets. We believe that the method can be applied to the general case of $k$-decompositions, but the proof will be much more technical. However, for $k\ge 4$, we have a way to decompose $I_r$ into $k$ MSTD subsets by simply using $2$-decompositions, which will be discussed later.


\subsection{Explicit Construction of Infinite Families of MSTD sets}
The following lemma is useful in proving many of our results.

\begin{lem}\label{P_nsum}
Let $A=L\cup R$ be an MSTD, $P_n$ set containing $1$ and $2n$, where $L\subseteq [1,n]$ and $R\subseteq [n+1,2n]$. Form $A'=L\cup M\cup R'$, where $M\subseteq [n+1,n+m]$ and $R'=R+m$ for some $m\in\mathbb{N}_0$. If $A'$ is a $SP_n$ set, then $A'$ is MSTD.
\end{lem}

\begin{proof}
We prove that $A'$ is MSTD by showing that the increase in the number of differences is at most the increase in the number of sums. As shown in the proof of Lemma 2.1 in \cite{MOS}, the number of new added sums is $2m$. Because $R'=R+m$, all differences in $[-(2n+m-1),-(n+m)]$ can be paired up with differences in $[1-2n,-n]$ from $L-R$ and differences in $[n+m,2n+m-1]$ can be paired up with differences in $[n,2n-1]$ from $R-L$. Because the set $A$ is $P_n$, $A$ contains all numbers in $[-n+1,n-1]$. In the worst scenario (in terms of the increase in the number of differences), $A'-A'$ contains all differences in $[-(n+m)+1,(n+m)-1]$. So, at most $|A'-A'|-|A-A|=|[-(n+m)+1,(n+m)-1]|-|[-n+1,n-1]|=2m$. This completes our proof.
\end{proof}

\begin{cor}\label{fixlem}
Let $A = L\cup R$ be an MSTD, $P_{n-4}$ set containing $5$ and $2n-4$, where $L\subseteq [5,n]$ and $R\subseteq [n+1,2n-4]$. Form $A' = L\cup M\cup R'$, where $M\subseteq [n+1,n+m]$ and $R' = R+m$ for some $m\in \mathbb{N}_0$. If $A'$ is a $SP_{n-4}$ set, then $A'$ is MSTD.
\end{cor}

\begin{proof}
In Lemma \ref{P_nsum}, we use $n-4$ instead of $n$, then consider $4+A$.
\end{proof}

\begin{lem}\label{ft}
Let an MSTD, $P_n$ set $A$ be chosen, where $A=L\cup R$ for $L\subseteq [1,n]$ and $R\subseteq [n+1,2n]$. Additionally, $L$ and $R$ must satisfy the following conditions: $[1,4]\cup\{n\}\subseteq L$ and $\{n+1\}\cup[2n-3,2n]\subseteq R$. Pick $k\ge n/2+2$ and $m\in\mathbb{N}_0$. Form \begin{align*}O_1\ &=\ \{n+4\}\cup[n+5,n+2k+1]_2\cup\{n+2k+2\}\\ O_2\ &=\ \{n+m+2k+3\}\cup[n+m+2k+4,n+m+4k]_2\cup\{n+m+4k+1\}.\end{align*} Let $M\subseteq [n+2k+3,n+m+2k+2]$ be such that within $M$, there exists a sequence of pairs of consecutive elements, where consecutive pairs in the sequence are not more than $2k-1$ apart and the sequence starts with a pair in $[n+2k+3,n+4k+1]$ and ends with a pair in $[n+m+4,n+m+2k+2]$. Denote $A'=L\cup O_1\cup M\cup O_2\cup R'$, where $R'=R+m+4k+4$. Then $A'$ is MSTD.
\end{lem}

\begin{proof}
We know that $A'\subseteq [2,4n+2m+8k+8]$. To prove that $A'$ is MSTD, it suffices to prove that $A'$ is $SP_n$. In particular, we want to show that $[n+2,3n+2m+8k+8]\subseteq A'+A'$. Due to symmetry\footnote{Due to symmetry, $A'$ has the same structure as $(2n+m+4k+5)-A'$. If $[n+2,2n+m+4k+5]\subseteq A'+A'$, then $[n+2,2n+m+4k+5]\subseteq (2n+m+4k+5-A')+(2n+m+4k+5-A')=(4n+2m+8k+10)-(A'+A')$ and so, $[2n+m+4k+5,3n+2m+8k+8]\subseteq A'+A'$.}, it suffices to show that $[n+2,2n+m+4k+5]\subseteq A'+A'$. We have:
\begin{align*}
[n+2,n+4]\ &\subseteq \ A'+A'\quad (\text{because }2,3,4,n\in A')\\
(1+O_1)\cup (2+O_1)\ &=\ [n+5,n+2k+4]\\
O_1+O_1\ &=\ [2n+8,2n+4k+4].\\
\end{align*}
Since $n+2k+4\ge 2n+8$, $[n+2,2n+4k+4]\subseteq A'+A'$. Consider $M+O_1$. In the worst scenario (in terms of getting necessary sums), the two smallest elements of $M$ are $n+4k$ and $n+4k+1$, while the two largest elements of $M$ are $n+m+4$ and $n+m+5$. We have $M+O_1\supseteq [2n+4k+4,2n+m+2k+7]$. We complete the proof by showing that $[2n+m+2k+8,2n+m+4k+5]\subseteq A'+A'$. We have $((n+4)+O_2)\cup ((n+5)+O_2)=[2n+m+2k+7,2n+m+4k+6]$. So, $A'$ is $SP_n$ and thus, MSTD by Lemma \ref{P_nsum}.
\end{proof}

\begin{lem}\label{ss}
Let an MSTD, $P_{n-4}$ set $A$ be chosen, where $A=L\cup R$ for $L\subseteq [5,n]$ and $R\subseteq [n+1,2n-4]$. Additionally, $L$ and $R$ must satisfy the following conditions: $[5,7]\subseteq L$ and $[2n-6,2n-4]\subseteq R$. Pick $k\ge n/2-5$ and $m\in\mathbb{N}_0$. Form \begin{align*}
O_1\ &=\ [n+1,n+3]\cup[n+6,n+2k]_2\cup[n+2k+3,n+2k+5],\\
O_2\ &=\ [n+m+2k,n+m+2k+2]\cup[n+m+2k+5,n+m+4k-1]_2\\
&\qquad\qquad\cup[n+m+4k+2,n+m+4k+4].
\end{align*}
Let $M\subseteq [n+2k+6,n+m+2k-1]$ such that within $M$, there exists a sequence of triplets of consecutive elements, where consecutive triplets in the sequence are not more than $2k+5$ apart and the sequence starts with a triplet in $[n+2k+6,n+4k+5]$ and ends with a triplet in $[n+m,n+m+2k-1]$. Denote $A'=L\cup O_1\cup M\cup O_2\cup R'$,
where $R'=R+m+4k+4$. Then $A'$ is MSTD.
\end{lem}

\begin{proof}
We want to show that $A'$ is $SP_{n-4}$. By Corollary \ref{fixlem}, we know that $A'$ is MSTD. In particular, we prove that $[n+6, 3n+2m+8k+4]\in A'+A'$. It suffices to prove that $[n+6,2n+m+4k+5]\subseteq A'+A'$.\footnote{Due to symmetry, $A'$ has the same structure as $(2n+m+4k+5)-A'$. If $[n+6,2n+m+4k+5]\subseteq A'+A'$, then $[n+6,2n+m+4k+5]\subseteq (2n+m+4k+5-A')+(2n+m+4k+5-A')=(4n+2m+8k+10)-(A'+A')$ and so, $[2n+m+4k+5,3n+2m+8k+4]\subseteq A'+A'$.} We have
\begin{align*}
(5+O_1)\cup (6+O_1)\cup (7+O_1)\ &=\ [n+6,n+2k+12]\\
O_1+O_1\ &=\ [2n+2,2n+4k+10].
\end{align*}
Because $n+2k+12\ge 2n+2$, $A'+A'$ contains $[n+6,2n+4k+10]$. Consider $M+O_1$. In the worst scenario (in terms of getting sums), the smallest elements in $M$ are $n+4k+3, n+4k+4$ and $n+4k+5$, while the largest elements in $M$ are $n+m,n+m+1$ and $n+m+2$. Then $M+O_1\supseteq [2n+4k+4,2n+m+2k+7]$. We only need to show that $[2n+m+2k+8,2n+m+4k+5]\subseteq A'+A'$. We have $((n+1)+O_2)\cup ((n+2)+O_2)\cup ((n+3)+O_2)=[2n+m+2k+1,2n+m+4k+7]$. This completes our proof that $A'$ is $SP_{n-4}$ and thus, MSTD.
\end{proof}

\subsection{Explicit Partitions into Two MSTD Sets}
We are now ready to prove Theorem \ref{2decomp}. The proof follows from Lemma \ref{ft} and Lemma \ref{ss}.

\begin{proof}
As indicated in Remark \ref{Exp1}, there exist pairs of sets $A_1$ and $A_2$ such that all conditions in Theorem \ref{2decomp} are satisfied. Pick $k\ge n/2+2$ and $m$ sufficiently large. Set $R_i'=R_i+m+4k+4$ for $i=1,2$. Form
\begin{align*}
O_{11}\ &=\ \{n+4\}\mbox{ }\cup\mbox{ }[n+5,n+2k+1]_2\mbox{ }\cup\mbox{ }\{n+2k+2\},\\
O_{12}\ &=\ \{n+m+2k+3\}\mbox{ }\cup\mbox{ }[n+m+2k+4,n+m+4k]_2\mbox{ }\cup\mbox{ }\{n+m+4k+1\},\\
O_{21}\ &=\ [n+1,n+3]\mbox{ }\cup\mbox{ }[n+6,n+2k]_2\mbox{ }\cup\mbox{ }[n+2k+3,n+2k+5],\\
O_{22}\ &=\ [n+m+2k,n+m+2k+2]\mbox{ }\cup\mbox{ }[n+m+2k+5,n+m+4k-1]_2\\
&\qquad\qquad\mbox{ }\cup\mbox{ }[n+m+4k+2,n+m+4k+4].
\end{align*}
We see that $O_{11}\cup O_{21}=[n+1,n+2k+5]$ and $O_{12}\cup O_{22}=[n+m+2k,n+m+4k+4]$. By Lemma \ref{ft} and Lemma \ref{ss}, we know that $A_1'=L_1\cup O_{11}\cup M_1\cup O_{12}\cup R'_1$ and $A_2'=L_2\cup O_{21}\cup M_2\cup O_{22}\cup R'_2$ are MSTD sets and $(A_1',A_2')$ partitions $[1,2n+m+4k+4]$, given that the following three conditions are satisfied:
\begin{enumerate}
\item $M_1\subseteq [n+2k+6,n+m+2k-1]$ such that within $M_1$, there exists a sequence of pairs of consecutive elements, where consecutive pairs in the sequence are not more than $2k-1$ apart and the sequence starts with a pair in $[n+2k+6,n+4k+1]$ and ends with a pair in $[n+m+4,n+m+2k-1]$;
\item $M_2\subseteq [n+2k+6,n+m+2k-1]$ such that within $M_2$, there exists a sequence of triplets of consecutive elements, where consecutive triplets in the sequence are not more than $2k+5$ apart and the sequence starts with a triplet in $[n+2k+6,n+4k+5]$ and ends with a triplet in $[n+m,n+m+2k-1]$; and
\item $M_1\cup M_2=[n+2k+6,n+m+2k-1]$ and $M_1\cap M_2=\emptyset$.
\end{enumerate}
For $m$ sufficiently large, it is obvious that $M_1$ and $M_2$ exist. This completes the proof of the theorem.
\end{proof}

\begin{rek}\normalfont
Observe that our fringe pairs in this case are $(L_1\cup O_{11},R_1'\cup O_{12})$ and $(L_2\cup O_{21}, R_2'\cup O_{22})$. Though disjoint, the union of the two fringe pairs gives us a full set of left and right elements of $[1,2n+m+4k+4]$ and each is a fringe pair for an MSTD set.
\end{rek}


\section{Explicit $k$-decomposition into MSTD Subsets}\label{explicitk}


\subsection{Overview}\label{overview}
Theorem \ref{2decomp} gives us a way to partition $I_r$ into two MSTD subsets. Due to linear transformations, we can partition any (long enough) arithmetic progressions into two MSTD subsets. If we can find an MSTD subset $S$ of $I_r$ such that $I_r\backslash S$ is a union of $k$ arithmetic progressions $\cup_{j=1}^kP_j$, then we can partition $I_r$ into $1+2k$ MSTD subsets (because each $P_j$ can be partitioned into two MSTD subsets). This is the central idea in both methods we use to $k$-decompose $I_r$ into MSTD subsets presented later.


\subsection{Base Expansion Method ($k\ge 3$)}
We explicitly provide a way to $k$-decompose $I_r$ into MSTD subsets. First, we need to define a ``strong MSTD'' set.
We call a set $S$ a $10$-strong MSTD set if $|S+S|-|S-S|\ge 10|S|$.\footnote{We pick the number 10 just to be safe for our later arguments. We make no attempt to provide an efficient way to decompose $I_r$ into $k$ MSTD subsets. 
}
This type of MSTD set does exist. For example, using the base expansion method, we can construct such a set. Using $\tilde{S}=\{0,2,3,4,7,11,12,14\}$, by the method, we can construct $S$ such that $|S|=|\tilde{S}|^2=8^4=4096$, $|S+S|=26^4=456976$ and $|S-S|=25^4=390625$; then, $|S+S|-|S-S|>10|S|$.

\begin{lem} \label{strongMSTDplus}
If $S$ is a $10$-strong MSTD set, then $S\cup\{a_1,a_2,a_3,a_4\}$, where $a_4>a_3>a_2>a_1>\max S$, is an MSTD set. Similarly, $S\cup\{b_1,b_2,b_3,b_4\}$, where $b_1<b_2<b_3<b_4<\min S$, is also an MSTD set.
\end{lem}

\begin{proof}
We want to show that $S\cup\{a_1,a_2,a_3,a_4\}$ is MSTD. Adding one more element to a set $S$ produces at least $0$ new sums and at most $2|S|$ new differences. So, $|(S\cup\{a_1\})+(S\cup\{a_1\})|-|(S\cup\{a_1\})-(S\cup\{a_1\})|\ge 10|S|+0-2|S|=8|S|$. So, $S\cup\{a_1\}$ is MSTD. Define $S_1=S\cup\{a_1\}$ with $|S_1|=|S|+1$. Similarly, $|(S_1\cup\{a_2\})+(S_1\cup\{a_2\})|-|(S_1\cup\{a_2\})-(S_1\cup\{a_2\})|\ge 8|S|+0-2|S_1|=8|S|-2(|S|+1)=6|S|-2$. Again, $S_1\cup\{a_2\}$ is MSTD. Repeating this argument, we can show that $S_4=S\cup\{a_1,a_2,a_3,a_4\}$ is an MSTD set. The proof is similar for $S\cup\{b_1,b_2,b_3,b_4\}$. Note that in the end, we reach the requirement that $|S|>6$, which is certainly true by \cite[Theorem 1]{He}.
\end{proof}

\begin{rek} \label{MSTDinarithmeticprog}\normalfont
The following fact will be useful later: An arithmetic progression of integers (assumed long enough) can contain an arbitrarily large number of disjoint $10$-strong MSTD sets.
\end{rek}

With the above remark, we are ready to prove the following.

\begin{lem}
There exists $N\in\mathbb{N}$ such that for all $r\ge N$, $I_r$ can be partitioned into exactly three MSTD subsets.
\end{lem}

\begin{proof}
We use a pair of fringe elements described in \cite{MO}: $L=\{1,3,4,8,9,10,11\}$ and $R=\{r-10,r-9,r-8,r-7,r-5,r-2,r-1,r\}$. We see that $I_r\backslash (L\cup R)=\{2,5,6,7\}\cup[12,r-11]\cup \{r-3,r-4,r-6\}$. We have \begin{align*}
L+L&\ =\ [2,22]\setminus \{3\}\\
L+R&\ =\ [r-9,r+11]\\
R+R&\ =\ [2r-20,2r].
\end{align*}
Consider $K=\{\ell\,|\,12\le \ell\le r-11,\ell\mbox{ is even}\}\cup\{r-11\}$. We have $(\{11\}\cup K\cup\{r-10\})+(\{11\}\cup K\cup\{r-10\})=[22,2r-20]$. So, $(L\cup K\cup R)+(L\cup K\cup R)=[2,2r]\backslash\{3\}$. Because $\pm(R-L)$ lacks $\pm(r-7)$, $L\cup K\cup R$ is an MSTD set. It is not hard to see that adding numbers in $[12,r-11]\backslash K$ to $L\cup K\cup R$ still gives an MSTD set.

Now, $[12,r-11]\backslash K$ contains an arithmetic progression of consecutive odd integers. We can make this arithmetic progression arbitrarily large by increasing $r$. By Remark \ref{MSTDinarithmeticprog}, this arithmetic progression can contain two disjoint $10$-strong MSTD sets, called $S_1$ and $S_2$. We write $[12,r-11]\backslash K=S_1\cup S_2\cup M$. By Lemma \ref{strongMSTDplus}, $S_1^*=S_1\cup\{2,5,6,7\}$ and $S_2^*=S_2\cup\{r-6,r-4,r-3\}$ are both MSTD. By what we say above, $K^*=M\cup L\cup K\cup R$ is also MSTD. Because $S_1^*\cup S_2^*\cup K^*=I_r$, we have completed the proof.
\end{proof}

\begin{proof}[Proof of item (1) of Theorem \ref{kdecomp}]
Let $k\ge 2$ be chosen. Write $k=2m_1+3m_2$ for some $m_1$ and $m_2\in\mathbb{N}_0$. We can find $N\in\mathbb{N}$ such that $I_N=[1,N]=\big([1,k_1]\cup[k_1+1,k_2]\cup\cdots\cup[k_{m_1-1}+1,k_{m_1}]\big)
\cup\big([k_{m_1}+1,k_{m_1+1}]\cup[k_{m_1+1}+1,k_{m_1+2}]\cup\cdots\cup[k_{m_1+m_2-1}+1,N]\big)$, such that each of the first $m_1$ intervals are large enough to be partitioned into two MSTD sets while the next $m_2$ intervals are large enough to be partitioned into three MSTD sets. So, $I_N$ can be partitioned in exactly $k$ MSTD sets. This completes our proof.
\end{proof}


\subsection{Efficient Methods $(k\ge 4)$}

\subsubsection{Notations and Preliminary Results}

We introduce a notation to write a set; this notation was first used by Spohn \cite{Sp}. Given a set $S=\{a_1,a_2,\ldots,a_n\}$, we arrange its elements in increasing order and find the differences between two consecutive numbers to form a sequence. Suppose that $a_1<a_2<\cdots<a_n$; then our sequence is $a_2-a_1,a_3-a_2,a_4-a_3,\ldots,a_n-a_{n-1}$. Then we represent $$S\ = \ (a_1|a_2-a_1,a_3-a_2,a_4-a_3,\ldots,a_n-a_{n-1}).$$ Take $S=\{3,2,5,10,9\}$, for example. We arrange the elements in increasing order to have $2,3,5,9,10$ and form a sequence by looking at the difference between two consecutive numbers: $1,2,4,1$. So, we write $S=(2|1,2,4,1)$. All information about a set is preserved in this notation.

\begin{lem}\label{threesetsforpartition}
The following are MSTD sets for a given $m\in\mathbb{N}$: \begin{align*}&(1|1,1,2,1,\underbrace{4\ldots,4}_{m\text{-times}},3,1,1,2),\\ &(1|1,1,2,1,\underbrace{4\ldots,4}_{m\text{-times}},3,1,1,2,1),\\ &(1|1,1,2,1,\underbrace{4\ldots,4}_{m\text{-times}},3,1,1).\end{align*}
\end{lem}

We present the proof in Appendix \ref{Apen34}.
\subsubsection{Efficient Methods ($k\ge 4$)}
In our decomposition of $I_r$ into $k$ MSTD subsets for $k\ge 3$, we use the base expansion method. However, the base expansion method is inefficient in terms of cardinalities of our sets. Is there a more efficient way to decompose? In answering this question, we present a method of decomposing $I_r$ into $k$ MSTD subsets ($k\ge 4$) that helps reduce the cardinalities of sets. We use the infinite family of MSTD sets in Lemma \ref{threesetsforpartition} to achieve this.

We want to decompose $I_r$ for sufficiently large $r$ into $k$ ($k\ge 4$) MSTD subsets. If $k$ is even, we can simply write $I_r$ as the union of $k/2$ arithmetic progressions, each of which, by Theorem \ref{2decomp}, can be decomposed into two MSTD subsets in an efficient way. If $k\ge 5$ is odd, then we consider $r\Mod 4$. If $r\equiv 1\Mod 4$, write $r=4m+13$ for some $m\in\mathbb{N}$ and consider $(1|1,1,2,1,\underbrace{4\ldots,4}_{m\text{-times}},3,1,1,2)$. We have \begin{equation*}\begin{split}&I_r\backslash\{1,2,3,5,6,10,14,18,22,26,\ldots,6+4m,9+4m,10+4m,11+4m,13+4m\}\\
&=\{4,8,12,16,20,\ldots,8+4m,12+4m\}\cup\{7,9,11,13,\ldots,7+4m\}.\end{split}\end{equation*} Notice that both $\{4,8,12,16,20,\ldots,8+4m,12+4m\}$ and $\{7,9,11,13,\ldots,7+4m\}$ are arithmetic progressions and each of these sets can be decomposed into an even number of MSTD sets. So, our original sets $I_r=[1,13+4m]$ can be decomposed into exactly $k$ MSTD sets. If $r\equiv 2\Mod 4$, write $r=4m+14$ and consider $(1|1,1,2,1,\underbrace{4\ldots,4}_{m\text{-times}}$, $3,1,1,2,1)$. If $r\equiv 3\Mod 4$, write $r=4m+11$ and consider $(1|1,1,2,1,\underbrace{4\ldots,4}_{m\text{-times}},3,1,1)$. If $r\equiv 0\Mod 4$, write $r=4m+12$ and consider $(2|1,1,2,1,\underbrace{4\ldots,4}_{m\text{-times}},3,1,1)$. Using the same argument as above, we can show that $I_r$ can be decomposed into exactly $k$ MSTD sets. We prove the upper and lower bounds for $R(k)$ in Appendix \ref{ApenR}.


\section{Future Work}
We end with several additional questions to pursue.

\begin{enumerate}
    \item In \cite{AMMS}, the authors show that there is a positive constant lower bound for the percentage of decompositions into two MSTD sets. Is there a positive constant lower bound for the percentage of decompositions into $k$ MSTD sets for $k\ge 3?$ In other words, is Conjecture \ref{conjallk} true? A method is to find a family of sets $(A_i)_{i=1}^{k}$ that satisfies the condition in Theorem \ref{sufcon}.
    \item Is there a method of $k$-decomposition that is of high density, for example $\Theta (1/r^c)$ for small $c?$
    \item For the $3$-decomposition, we use the base expansion method, which is inefficient. Can we find an efficient way to decompose $I_r$ into three MSTD subsets.
    \item Can we find some better bounds for $R(k)$ in Theorem \ref{kdecomp}? There is a yawning gap between our upper and lower bounds.
    \item Suppose that $I_r$ can be decomposed into $k$ MSTD subsets. Can we conclude that $I_{r+1}$ can be decomposed into $k$ MSTD subsets?

\end{enumerate}

\appendix

\section{Lower and Upper Bounds for $R(k)$ in Theorem \ref{kdecomp}}\label{ApenR}
Given $k\ge 2$, the lower bound is obvious since by \cite{He}, the smallest cardinality of an MSTD set is 8.
In \cite{AMMS}, it is shown that for all $r\ge 20$, $I_r$ can be partitioned into two MSTD subsets.
To decompose $I_r$ into $k\ge 2$ (even) MSTD subsets, we write $I_r$ to be the union of $k/2$ arithmetic progressions and require each to be of length at least 20. So, for all $r\ge 10k$, $I_r$ can be decomposed into $k$ MSTD subsets. Hence, $R(k)\le 10k$.

In our method to decompose $I_r$ into $k\ge 5$ (odd) MSTD subsets, we use particular MSTD sets, which are
\begin{align*}
A_1\ &=\ (1|1,1,2,1,\underbrace{4\ldots,4}_{m\text{-times}},3,1,1,2),\\
A_2\ &=\ (1|1,1,2,1,\underbrace{4\ldots,4}_{m\text{-times}},3,1,1,2,1),\\
A_3\ &=\ (1|1,1,2,1,\underbrace{4\ldots,4}_{m\text{-times}},3,1,1),\\
A_4\ &=\ (2|1,1,2,1,\underbrace{4\ldots,4}_{m\text{-times}},3,1,1).
\end{align*}
These sets have the property that $I_{\max A_i}\backslash A_i,i\in [1,4]$ is the union of two arithmetic progressions. Given $m$, $I_{\max A_3}\backslash A_3$ gives a pair of arithmetic progressions of shortest length, $m+2$ and $2m+1$, while $\max A_2=4m+14=\max\{\max A_i|i\in[1,4]\}$. We consider two cases.
\begin{enumerate}
    \item $k=4j+1\mbox{ } (j\ge 1)$. We require that all arithmetic progressions of length at least $m+2$ can be partitioned into $2j$ MSTD sets. Then $m+2\ge 20j$ and so, $m\ge 20j-2$, which also guarantees that all arithmetic progressions of length at least $2m+1$ can be partitioned into $2j$ MSTD sets. So, we find out that for $r\ge 4(20j-2)+14=20k-14$, $I_r$ can be partitioned into $k$ MSTD subsets. Hence, $R(k)\le 20k-14$.
    \item $k=4j+3\mbox{ } (j\ge 1)$. We require that all arithmetic progressions of length at least $m+2$ can be partitioned into $2j$ MSTD sets. Then $m+2\ge 20j$ and so, $m\ge 20j-2$, which also guarantees that all arithmetic progressions of length at least $2m+1$ can be partitioned into $2j+2$ MSTD sets. So, we find out that for $r\ge 4(20j-2)+14=20k-54$, $I_r$ can be partitioned into $k$ MSTD subsets. Hence, $R(k)\le 20k-54$.
\end{enumerate}
Finally, for $3$-decomposition, we use the base expansion method, where we require a run of consecutive odd numbers (an arithmetic progression) to contain two disjoint $10$-strong MSTD sets. The length of the arithmetic progression is at least $\frac{r-26}{2}+1$. Let $T$ be $\min\{\max A: A \mbox{ is }10-\mbox{strong }\}$. Then we require $\frac{r-26}{2}+1\ge 2T$ or $r\ge 4T+24$. Hence, $24\le R(k)\le 4T+24$.



\section{Sufficient Condition for a Positive Constant Bound}\label{Apensuf}

\begin{lem}\label{MOmodi}
Consider $S\subseteq \{0,1,\ldots,r-1\}$ and $S=L\cup M\cup R$. Fix $L\subseteq [0,\ell-1]$ and $R\subseteq [r-\ell,r-1]$ for some fixed $\ell$. Let $M$ be a uniformly randomly chosen subset of $[\ell, r-\ell-1]$. Then for any $\varepsilon>0$, there exists sufficiently large $r$ such that
\begin{equation}
\mathbb{P}([2\ell-1,2r-2\ell-1]\subseteq S+S)\ \ge \ 1-6(2^{-|L|}+2^{-|R|})-\varepsilon.
\end{equation}
\end{lem}

\begin{proof}
We write
\begin{equation}\begin{split}&\mathbb{P}([2\ell-1,2r-2\ell-1]\subseteq S+S)\ = \ 1-\mathbb{P}([2\ell-1,2r-2\ell-1]\not\subseteq S+S)\\
                                                   &= \ 1-\mathbb{P}([2\ell-1,r-\ell-1]\cup [r+\ell-1,2r-2\ell-1]\not\subseteq S+S\\&\indent \mbox{ or } [r-\ell,r+\ell-2]\not\subseteq S+S)\\
                                                   &\ge \ 1-\mathbb{P}([2\ell-1,r-\ell-1]\cup [r+\ell-1,2r-2\ell-1]\not\subseteq S+S\\
                                                   &\indent -\mathbb{P}([r-\ell,r+\ell-2]\not\subseteq S+S).
\end{split}
\end{equation}

By Proposition 8 in \cite{MO}, \begin{equation}\mathbb{P}([2\ell-1,r-\ell-1]\cup [r+\ell-1,2r-2\ell-1]\not\subseteq S+S)\ \le \ 6(2^{-|L|}+2^{-|R|}).\end{equation} We find a upper bound for
\begin{equation}\begin{split}\mathbb{P}([r-\ell,r+\ell-2]\not\subseteq S+S)\ \le \ \mathbb{P}([r-\ell,r+\ell-2]\not\subseteq M+M),\end{split}\end{equation} because $[r-2\ell,r-2]\not\subseteq S+S$ implies $[r-\ell,r+\ell-2]\not\subseteq M+M$. By a linear shift of $\ell$, we can consider $M$ a subset of $[0,r-2\ell-1]$ and $\mathbb{P}([r-\ell,r+\ell-2]\not\subseteq M+M)$ turns into $\mathbb{P}([r-2\ell,r-2]\not\subseteq M+M)$. Use the change of variable $N=r-2\ell$. We have: $M\subseteq [0,N-1]$ and we estimate:
$\mathbb{P}([r-2\ell,r-2]\not\subseteq M+M)=\mathbb{P}([N,N+2\ell-2]\not\subseteq M+M)\le \sum_{k=N}^{N+2\ell-2}\mathbb{P}(k\notin M+M)$. Lemma 7 in \cite{MO} shows that the last quantity tend to $0$ as $N$ goes to infinity. So, for any $\varepsilon>0$, there exists sufficiently large $r$ such that $\mathbb{P}([r-\ell,r+\ell-2]\not\subseteq S+S)<\varepsilon$. This completes our proof.
\end{proof}

\begin{lem}\label{taudefi}
Consider $S\subseteq \{0,1,\ldots,r-1\}$ and $S=L\cup M\cup R$. Fix $L\subseteq [0,\ell-1]$ and $R\subseteq [r-\ell,r-1]$ for some fixed $\ell$. Let $M$ be a uniformly randomly chosen subset of $[\ell,r-\ell-1]$. Let $a$ denote the smallest integer such that both $[\ell,2\ell-a]\subseteq L+L$ and $[2r-2\ell+a-2,2r-\ell-2]\subseteq R+R$. Then, for all $\varepsilon>0$, there exists sufficiently large $r$ such that
\begin{equation}
\begin{split}
\mathbb{P}&([2\ell-a+1,2r-2\ell+a-3]\subseteq S+S)\\ &\ge \ 1-(a-2)(2^{-\tau(R)}+2^{-\tau(L)})-6(2^{-|L|}+2^{-|R|})-\varepsilon,
\end{split}
\end{equation}
where $\tau(L)=|\{i\in L|i\le \ell-a+1\}|$ and $\tau(R)=|\{i\in R|i\ge r-\ell+a-2\}|$.
\end{lem}

\begin{proof}
We have:
\begin{align*}
&\mathbb{P}([2\ell-a+1,2r-2\ell+a-3]\subseteq S+S)\\
&= \ \mathbb{P}([2\ell-1,2r-2\ell-1]\subseteq S+S\mbox{ and } [2\ell-a+1,2\ell-2]\subseteq S+S\\&\qquad\qquad\mbox{ and }[2r-2\ell,2r-2\ell+a-3]\subseteq S+S)\\
&= \ 1- \mathbb{P}([2\ell-1,2r-2\ell-1]\not\subseteq S+S\mbox{ or } [2\ell-a+1,2\ell-2]\not\subseteq S+S\\&\qquad\qquad\mbox{ or }[2r-2\ell,2r-2\ell+a-3]\not\subseteq S+S)\\
&\ge \ 1-\mathbb{P}([2\ell-a+1,2\ell-2]\not\subseteq S+S)-\mathbb{P}([2r-2\ell,2r-2\ell+a-3]\not\subseteq S+S)\\&\quad -\mathbb{P}([2\ell-1,2r-2\ell-1]\not\subseteq S+S).
\end{align*}
By Lemma \ref{MOmodi}, $\mathbb{P}([2\ell-1,2r-2\ell-1]\not\subseteq S+S)\le 6(2^{-|L|}+2^{-|R|})+\varepsilon$. We have
\begin{align}
\mathbb{P}([2\ell-a+1,2\ell-2]\not\subseteq S+S)\ \le \ \sum_{k=2\ell-a+1}^{2\ell-2}\mathbb{P}(k\notin S+S).
\end{align}
Let $\tau(L)=|\{i\in L|i\le \ell-a+1\}|$ and $\tau(R)=|\{i\in R|i\ge r-\ell+a-2\}|$. For each value of $k$ in $[2\ell-a+1,2\ell-2]$, in order that $k\notin S+S$, all pairs of numbers that sum up to $k$ must not be both in $S$. Take $k=2\ell-a+1$, for example. For a number $x\le \ell-a+1$, the number $y$ that when added to $x$ gives $2\ell-a+1$ is at least $\ell$ and $y\notin S$. So, $\mathbb{P}(k\notin S+S)\le 2^{-\tau(L)}$ and hence,
\begin{align}
\mathbb{P}([2\ell-a+1,2\ell-2]\not\subseteq S+S)\ \le \ (a-2)2^{-\tau(L)}.
\end{align}
Similarly,
\begin{align}
\mathbb{P}([2r-2\ell,2r-2\ell+a-3]\not\subseteq S+S)\ \le \ (a-2)2^{-\tau(R)}.
\end{align}
We have shown that
\begin{multline*}
	\mathbb{P}([2\ell-a+1,2r-2\ell+a-3]\subseteq S+S)\\
	\ge \ 1-(a-2)(2^{-\tau(R)}+2^{-\tau(L)})-6(2^{-|L|}+2^{-|R|}))-\varepsilon.
\end{multline*}
This completes our proof.
\end{proof}

\begin{cor}\label{coradjust}
Consider set $A=L\cup R$, where $L\subseteq [0,n-1]$ and $R\subseteq [n,2n-1]$. Let $m\in\mathbb{N}$ be chosen. Set $R'=R+m\subseteq [n+m,2n+m-1]$. Let $M\subseteq [n,n+m-1]$ be chosen uniformly at random. Form $S = L\cup M\cup R'$. Let $a$ be the smallest integer such that $[n,2n-a]\subseteq L+L$ and $[2n+a-2,3n-2]\subseteq R+R$. Then for all $\varepsilon>0$, there exists sufficiently large $m$ such that
\begin{equation}
\begin{split}
\mathbb{P}&([2n-a+1,2n+2m+a-3]\not\subseteq S+S)\\&\le \ (a-2)(2^{-\tau(R)}+2^{-\tau(L)})+6(2^{-|L|}+2^{-|R|})+\varepsilon,
\end{split}
\end{equation}
where $\tau(L)=|\{i\in L|i\le n-a+1\}|$ and $\tau(R)=|\{i\in R|i\ge n+a-2\}|$.
\end{cor}

\begin{proof}
The corollary follows immediately by setting $r=2n+m$ and $\ell=n$ in Lemma \ref{taudefi}. Also, notice that $R'$ is a linear shift of $R$.
\end{proof}

For conciseness, we denote $$f(L,R)\ =\ (a-2)(2^{-\tau(R)}+2^{-\tau(L)})+6(2^{-|L|}+2^{-|R|}).$$

\begin{thm}\label{sufcon}
Suppose that there exist sets $(A_i)_{i=1}^{k}$, which are pairwise disjoint, MSTD, $P_n$ and $\cup_{i=1}^kA_i=[0,2n-1]$. In particular, each $A_i=L_i\cup R_i$, where $L_i\subseteq [0,n-1]$ and $R_i\subseteq [n,2n-1]$. Let $m\in\mathbb{N}_0$ be chosen. Form $R_i'=R_i+m$ and $(M_i)_{i=1}^{k}\subseteq [n,n+m-1]$ such that $(M_i)_{i=1}^k$ are pairwise disjoint and $\cup_{i=1}^k M_i=[n,n+m-1]$. Then the proportion of cases where all $S_i=L_i\cup M_i\cup R'_i$ are MSTD is bounded below by a positive constant if
\begin{equation}
1-\sum_{i=1}^{k}f(L_i,R_i)\ >\ 0,
\end{equation} for $m$ sufficiently large. In other words, there exists a positive constant lower bound for the proportion of $k$-decompositions into MSTD subsets as $m\rightarrow\infty$.
\end{thm}

\begin{proof}
Let $a_i$ be the corresponding $a$ value (defined in Lemma \ref{taudefi}) for $L_i$ and $R_i$. By Corollary \ref{coradjust}, for any $\varepsilon > 0$ and $m$ sufficiently large, the probability
\begin{equation}
    \begin{split}
        &\mathbb{P}(\forall i,S_i \mbox{ is MSTD})\ge \mathbb{P}(\forall i, S_i+S_i\supseteq [2n-a_i+1,2n+2m+a_i-3])\\
        &\ge \ 1-\mathbb{P}(\exists i, S_i+S_i\not\supseteq [2n-a_i+1,2n+2m+a_i-3])\\
        &> \ 1-\sum_{i=1}^{k} f(L_i,R_i) - \varepsilon>0.
    \end{split}
\end{equation}
The first inequality is because $[2n-a_i+1,2n+2m+a_i-3]\subseteq S_i+S_i$ guarantees that $S_i$ is $SP_n$ and thus, MSTD. By Lemma \ref{P_nsum}, $S_i$ is MSTD.
\end{proof}

Now, we prove Theorem 1.4 in \cite{AMMS} easily.

\begin{cor}
There exists a constant $c>0$ such that the proportion of 2-decompositions of $[0,r-1]$ into two MSTD subsets is at least $c$.
\end{cor}
\begin{proof}
Let
\begin{align*}
L_1\ &=\ \{0,1,2,3,7,8,10,12,13,14,19\},\\
R_1\ &=\ \{20,25,26,27,30,32,36,37,38,39\},\\
L_2\ &=\ \{4,5,6,9,11,15,16,17,18\},\\
R_2\ &=\ \{21,22,23,24,28,29,31,33,34,35\}.
\end{align*}
Notice that $n=20$. We find that $a_1=12$ and $a_2=4$. From that we calculate $\tau(L_1)=6,\tau(R_1)=6,\tau(L_2)=8$ and $\tau(R_2)=9$. So, $f(L_1,R_1)$ is less than 0.33, while $f(L_2,R_2)$ is less than 0.03 for $m$ sufficiently large. By Theorem \ref{sufcon}, we are done.
\end{proof}



\section{Proof of Lemma \ref{threesetsforpartition}}\label{Apen34}
We prove that for a fixed $m\in\mathbb{N}$, $S=(0|1,1,2,1,\underbrace{4\ldots,4}_{m\text{-times}}, 3,1,1,2)$
is MSTD. The proof for other sets in the lemma follows similarly.

Note that $\max S=12+4m$. We will prove that $|S+S|\ge 26+6m$. Since $S$ contains $0$, $1$ and $2$, if the difference between two numbers, say $x<y$, in $S$ is less than or equal to 3, then $S+S$ contains $[x,y]$. If $a,b\in S$ and $a-b=4$, then in the worst case (in term of cardinality of the sum set), $S+S$ does not contain $a-1$. So, for the interval $[0,12+4m]$, $S+S$ misses at most $m-1$ sums because there are $m$ differences of $4$ and $8=4+4\in S+S$. Next, consider $[13+4m,24+8m]$ and observe that $S_1=\{\ell|1\le \ell\le 9+4m\mbox{ and }\ell \equiv 1\Mod 4\}\subseteq S$. Since $0\in S$, $S_1\in S+S$. We also have
\begin{align*}
    (12+4m)+S_1\ &=\ \{\ell|13+4m\le \ell\le 21+8m\mbox{ and }\ell \equiv 1\Mod 4\},\\
    (10+4m)+S_1\ &=\ \{\ell|11+4m\le \ell\le 19+8m\mbox{ and }\ell \equiv 3\Mod 4\},\\
    (9+4m)+S_1\ &=\ \{\ell|10+4m\le \ell\le 18+8m\mbox{ and }\ell \equiv 2\Mod 4\}.
\end{align*}
Note that
\begin{align*}
16+4m\ &= \ (12+4m)+4 \ \in\  S+S,\\
16+8m\ & =\ (8+4m)+(8+4m)\ \in\ S+S,\\
20+8m\ &= \ (10+4m)+(10+4m) \ \in\  S+S,\\
22+8m\ &= \ (10+4m)+(12+4m) \ \in\  S+S,\\
24+8m\ &= \ (12+4m)+(12+4m) \ \in\  S+S.
\end{align*}
On the interval $[13+4m, 24+8m]$, $S+S$ misses at most the whole set $\{\ell|20+4m\le \ell\le 12+8m\mbox{ and }\ell\equiv 0\Mod 4\}\cup\{23+8m\}$, which has $m$ numbers. Therefore, in total, $S+S$ misses at most $2m-1$ numbers.

Next, we show that the difference set $S-S$ misses at least $2m$ numbers by proving that $S-S$ contains none of the elements in $\{6+4\ell|0\le \ell\le m-1\}$. We use proof by contradiction. Suppose that there exists $0\le \ell\le m-1$ such that $6+4\ell$ is in $S-S$. Then, there must exist a run within $1,1,2,1,\underbrace{4\ldots,4}_{m\text{-times}},3,1,1,2$ that sums up to $6+4\ell$. Because $6+4\ell\equiv 2\Mod 4$, the run must either start within $1,1,2,1$ or end within $3,1,1,2$. Consider the following two cases:
\begin{enumerate}
    \item \textbf{Case I}: the run starts within $1,1,2,1$. Because $1+1+2+1=5<6$, the run must end within $3,1,1,2$. Therefore, the run sums up to a number of the form $a+4m+b$, where the value of $a$ and $b$ depend on where the run starts and where it ends, respectively. Since $a+4m+b=6+4\ell\le 6+4(m-1)$, $a+b\le 2$. This is a contradiction because $b\ge 3$.
    \item \textbf{Case II}: the run ends within $3,1,1,2$. Because there is no run within $\underbrace{4\ldots,4}_{m\text{-times}},\newline 3,1,1,2$ that sum up to $6+4\ell$, the run must start within $1,1,2,1$. Repeating the argument used in \textbf{Case I} and we have a contradiction.
\end{enumerate}
Therefore, $(S-S)\cap\{6+4\ell|0\le \ell\le m-1\}=\emptyset$ and so, $S-S$ misses at least $2m$ elements. This completes our proof that $S$ is MSTD. \hfill $\Box$


\section{Examples}\label{Apenex}


\subsection{Theorem \ref{2decomp}}\label{122}
We use $A_1$ and $A_2$ mentioned in Remark \ref{Exp1}. Pick $k=12$ and $m=30$. Set
\begin{align*}
A'_{1e}\ =\ &\{1,2,3,4,8,9,11,13,14,15,20\}\\
         &\cup \ \{24\}\ \cup\ [25,45]_2\ \cup\ \{46\}\cup \ \{50,51,72,73\}\\
         &\cup \ \{77\}\ \cup\ [78,98]_2\ \cup\ \{99\}\\
         &\cup \ \{103,108,109,110,113,115,119,120,121,122\},\\
A'_{2e}\ =\ &\{5,6,7,10,12,16,17,18,19\}\\
         &\cup \ [21,23]\ \cup\ [26,44]_2\ \cup\ [47,49]\cup\ \{52,53,54,69,70,71\}\\
         &\cup \ [74,76]\ \cup\ [79,97]_2\ \cup\ [100,102]\\
         &\cup\ \{104,105,106,107,111,112,114,116,117,118\}.
\end{align*}
We have $|A'_{1e}+A'_{1e}|-|A'_{1e}-A'_{1e}|=243-241=2$, $|A'_{2e}+A'_{2e}|-|A'_{2e}-A'_{2e}|=227-225=2$ and $(A'_{1e},A'_{2e})$ partitions $[1,122]$.


\subsection{5-decompositions}
We do not give an example of a 3-decomposition because our method is inefficient and involves a large set arising from the base expansion method. Neither do we give an example of a 4-decomposition because the method is straightforward. We use the efficient method to have a 5-decomposition into MSTD sets. Set
\begin{align*}
    M_1\ &=\ (1|1,1,2,1,\underbrace{4\ldots,4}_{119\text{-times}},3,1,1,2)\\
         &=\ \{1,2,3,5\}\ \cup\ [6,482]_4\ \cup\ \{485,486,487,489\}.
\end{align*}
Observe that \begin{align*}[1,489]\backslash M_1\ =\ &[4,488]_4\ \cup\ [7,483]_2.\end{align*}
Because $[4,488]_4$ is an arithmetic progression of length 122, we have \begin{align*}M_2\ &=\ 4A'_{1e}\mbox{ and}\\
    M_3\ &=\ 4A'_{2e}\end{align*}
partition $[4,488]_4$. Notice that in the example mentioned in \ref{122}, we can pick $m=147$ and find $A''_{1e}$ (containing 1) and $A''_{2e}$ that partition $[1,239]$. We have
\begin{align*}
    M_4\ &=\ 2A''_{1e}+5\mbox{ and} \\
    M_5\ &=\ 2A''_{2e}+5
\end{align*}
partition $[7,483]_2$. We have found $(M_1,M_2,M_3,M_4,M_5)$ that partitions $[1,489]$.



\subsection*{Acknowledgement.} We thank the referee for helpful comments on an earlier draft. We thank the participants from the 2018 SMALL REU program for many helpful conversations.

\bigskip

\end{document}